\documentclass{article}
\usepackage{amsmath}
\usepackage{amsfonts}
\usepackage{amsthm}
\usepackage{amssymb}
\usepackage{color}

\usepackage{graphicx}
\usepackage{float}
\usepackage{subfigure}
\usepackage[font=small]{caption}
\usepackage{tikz}

\newtheorem{theorem}{Theorem}[section]
\newtheorem{lemma}[theorem]{Lemma}
\newtheorem{proposition}[theorem]{Proposition}

\theoremstyle{definition}
\newtheorem{definition}[theorem]{Definition}
\theoremstyle{remark}
\newtheorem{remark}[theorem]{Remark}

\newcommand{\1}{\mathbf{1}}

\newcommand\remove[1]{}

\def\1{\mathbf{1}}

\def\f2{\mathbb{F}_2}

\def\cl{\hskip0.02cm{\rm cl}\hskip0.01cm}
\newcommand{\RNP}{{\rm RNP}}
\newcommand{\ESA}{{\rm ESA}}

\newcommand{\ep}{\varepsilon}

\newcommand{\sign}{{\rm sign}\hskip0.02cm}

\begin{document}

\title{Connections between metric characterizations of superreflexivity and the Radon-Nikod\'ym property for dual Banach spaces}

\author{Mikhail~I.~Ostrovskii\footnote{The author gratefully acknowledges the support by NSF
DMS-1201269. The author would like to thank William~B.~Johnson and
Beata~Randrianantoanina for useful
discussions related to the subject of this paper.}\\
\\
Department of Mathematics and Computer Science\\
St. John's University\\
8000 Utopia Parkway\\
Queens, NY 11439\\
USA\\
e-mail: {\tt ostrovsm@stjohns.edu} }

\date{\today}
\maketitle

\noindent{\sc Abstract:} Johnson and Schechtman (2009)
characterized superreflexivity in terms of finite diamond graphs.
The present author characterized the Radon-Nikod\'ym property
(RNP) for dual spaces in terms of the infinite diamond. This paper
is devoted to further study of relations between metric
characterizations of superreflexivity and the RNP for dual spaces.
The main result is that finite subsets of any set $M$ whose
embeddability characterizes the RNP for dual spaces, characterize
superreflexivity. It is also observed that the converse statement
does not hold, and that $M=\ell_2$ is a counterexample.\medskip

\noindent{\bf 2010 Mathematics Subject Classification:} 46B85
(primary), 46B07, 46B22 (secondary).

\begin{large}

\section{Introduction}

Results of \cite{JS09} and \cite{Ost14a} indicate the existence of
some parallels between metric characterizations of
superreflexivity and metric characterizations of dual spaces with
the Radon-Nikod\'ym property (RNP).

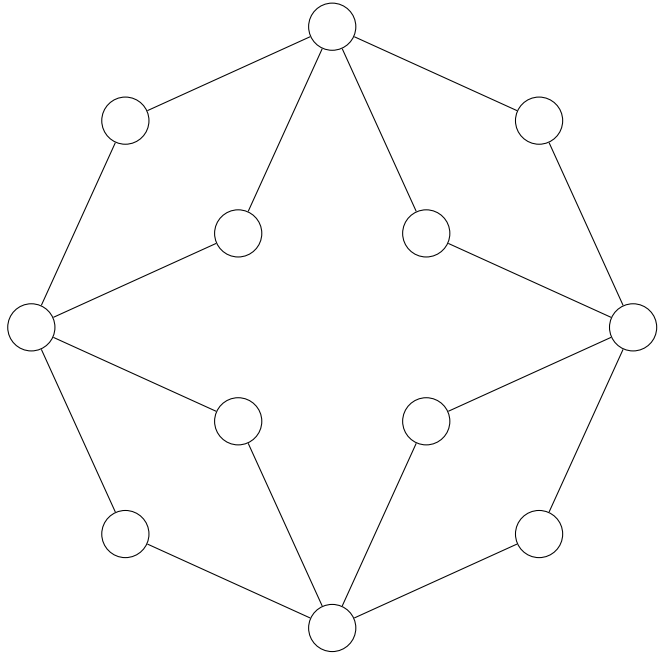
\begin{figure}
\begin{center}
\begin{tikzpicture}
  [scale=.25,auto=left,every node/.style={circle,draw}]
  \node (n1) at (16,0) {\hbox{~~~}};
  \node (n2) at (5,5)  {\hbox{~~~}};
  \node (n3) at (11,11)  {\hbox{~~~}};
  \node (n4) at (0,16) {\hbox{~~~}};
  \node (n5) at (5,27)  {\hbox{~~~}};
  \node (n6) at (11,21)  {\hbox{~~~}};
  \node (n7) at (16,32) {\hbox{~~~}};
  \node (n8) at (21,21)  {\hbox{~~~}};
  \node (n9) at (27,27)  {\hbox{~~~}};
  \node (n10) at (32,16) {\hbox{~~~}};
  \node (n11) at (21,11)  {\hbox{~~~}};
  \node (n12) at (27,5)  {\hbox{~~~}};

  \foreach \from/\to in {n1/n2,n1/n3,n2/n4,n3/n4,n4/n5,n4/n6,n6/n7,n5/n7,n7/n8,n7/n9,n8/n10,n9/n10,n10/n11,n10/n12,n11/n1,n12/n1}
    \draw (\from) -- (\to);

\end{tikzpicture}
 \caption{Diamond $D_2$.}\label{F:Diamond2}
\end{center}
\end{figure}

To state the corresponding results we recall the definition of the
infinite diamond. The {\it diamond graph} of level $0$ is denoted
$D_0$. It has two vertices joined by an edge of length $1$. $D_n$
is obtained from $D_{n-1}$ as follows. Each edge of $D_{n-1}$ is
of length $2^{-(n-1)}$. Given an edge $uv\in E(D_{n-1})$, it is
replaced by a quadrilateral $u, a, v, b$ with edge lengths
$2^{-n}$. We endow $D_n$ with their shortest path metrics. We
consider the vertex set of $D_n$ as a subset of the vertex set of
$D_{n+1}$, it is easy to check that this defines an isometric
embedding. We introduce $D_\omega$ as the union of the vertex sets
of $\{D_n\}_{n=0}^\infty$. For $u,v\in D_\omega$ we introduce
$d_{D_\omega}(u,v)$ as $d_{D_n}(u,v)$ where $n\in\mathbb{N}$ is
any integer for which $u,v\in V(D_n)$. Since the natural
embeddings $D_n\to D_{n+1}$ are isometric, $d_{D_n}(u,v)$ does not
depend on the choice of $n$ for which $u,v\in V(D_n)$. To the best
of my knowledge the first paper in which diamond graphs
$\{D_n\}_{n=0}^\infty$ were used in Metric Geometry is
\cite{GNRS04} (a conference version was published in 1999).

\begin{theorem}[\cite{JS09}]\label{T:JS}
A Banach space $X$ is nonsuperreflexive if and only if it admits
bilipschitz embeddings with uniformly bounded distortions of
diamonds $\{D_n\}_{n=1}^\infty$ of all sizes.
\end{theorem}

\begin{theorem}[\cite{Ost14a}]\label{T:RNPvsDmnd} A dual Banach space does not have the \RNP\ if and
only if it admits a bilipschitz embedding of $D_\omega$.
\end{theorem}

\begin{remark} It is known \cite{Ost11} that for Banach spaces which are
not dual spaces, lack of the \RNP\ does not imply embeddability of
$D_\omega$. (See \cite{Ost14a} for more results of this type.)
\end{remark}

Theorem \ref{T:JS} and \ref{T:RNPvsDmnd} make it natural to try to
understand whether similar results hold for other than $D_\omega$
separable metric spaces and their finite subsets. In this note we
prove that in one of the directions this is true. Recall (see
\cite{Ste75} and references therein) that a dual of a separable
Banach space has the \RNP\ if and only if it is separable. We
prove:

\begin{theorem}\label{T:InfToFin} If a metric space $M$ admits a bilipschitz embedding into any nonseparable
dual of a separable Banach space, then all of its finite subsets
embed into an arbitrary non-superreflexive Banach space with
uniformly bounded distortions.
\end{theorem}

The implication in the other direction does not hold in general.
We mean the following result which is proved in Section
\ref{S:HilbEx}.

\begin{proposition}\label{P:FinToInf} There exist a separable metric space $M$ and a separable Banach space $X$ with nonseparable dual $X^*$,
such that finite subsets of $M$ admit embeddings into an arbitrary
non-superreflexive Banach space with uniformly bounded
distortions, but $M$ does not admit a bilipschitz embedding into
$X^*$.
\end{proposition}

The Hilbert space $\ell_2$ is an example of an $M$ satisfying the
conditions of Proposition \ref{P:FinToInf}.

\begin{remark} It is worth mentioning that the Hilbert space is, up to an isomorphism, the only Banach space
finite subsets of which admit embeddings into an arbitrary
non-superreflexive Banach space with uniformly bounded
distortions. In fact, by results of James \cite{Jam78} and
Pisier-Xu \cite{PX87} there exist nonsuperreflexive spaces of type
2. It is well known that there exist nonsuperreflexive spaces of
cotype 2 (for example, $\ell_1$). On the other hand, Bourgain's
discretization theorem \cite{Bou87,GNS12} implies that uniform
bilipschitz embeddability of finite subsets implies existence of
uniformly isomorphic embeddings of finite-dimensional subspaces.
Therefore each Banach space satisfying the conditions of
Proposition \ref{P:FinToInf} has type 2 and has cotype 2, hence,
by the Kwapie\'n theorem \cite{Kwa72}, it is isomorphic to a
Hilbert space.
\end{remark}

Theorem \ref{T:InfToFin} is an immediate consequence of the
following result which is proved in the next section.

\begin{definition}\label{D:FinRep} Let $X$ and $Y$ be two Banach spaces. The space $X$
is said to be {\it finitely representable}  in $Y$ if for any
$\ep>0$ and any finite-dimensional subspace $F\subset X$ there
exists a finite-dimensional subspace $G\subset Y$ such that
$d(F,G)<1+\ep$, where $d(F,G)$ is the Banach-Mazur distance.

The space $X$ is said to be {\it crudely finitely representable}
in $Y$ if there exists $1\le C<\infty$ such that for any
finite-dimensional subspace $F\subset X$ there exists a
finite-dimensional subspace $G\subset Y$ such that $d(F,G)\le C$.
\end{definition}

\begin{theorem}\label{T:RNPandSuperRefl} For each non-superreflexive
Banach space $X$ there exists a non-separable dual $Z^*$ of a
separable Banach space $Z$, such that $Z^*$ is crudely finitely
representable in $X$.
\end{theorem}

We refer to \cite{BL00,LT73,LT77,Ost13,Pis14} for background
material and presentations of some of the results used below.

\section{Proof of Theorem \ref{T:RNPandSuperRefl}}

First we consider the case where $X$ has no nontrivial type. In
such a case $\ell_1$ is finitely representable in $X$ (by the
result of \cite{Pis73}), and therefore $Z=C(0,1)$ satisfies the
conditions of Theorem \ref{T:RNPandSuperRefl}. In fact, it is
clear that $(C(0,1))^*$ is nonseparable. It is also known (see
e.g. \cite[Section 5.b]{LT73}) that $(C(0,1))^*$ is finitely
representable in $\ell_1$.
\medskip

Now we consider the case where $X$ has nontrivial type. Replacing,
if necessary, $X$ by a nonreflexive space finitely represented in
it, we may assume that $X$ is nonreflexive. The following notion,
introduced by Brunel and Sucheston turned out to be a very useful
in the study of nonreflexive spaces with nontrivial type.

\begin{definition}[{\cite[p.~84]{BS75}}] A sequence $\{e_n\}$ in a semi-normed
space is called {\it equal signs additive} (ESA) if for any
finitely non-zero sequence $\{a_i\}$ of real numbers such that
$\sign a_k=\sign a_{k+1}$, the equality
\begin{equation}\label{E:ESA}
\left\|\sum_{i=1}^{k-1}a_ie_i+(a_k+a_{k+1})e_k+\sum_{i=k+2}^\infty
a_ie_i\right\|=\left\|\sum_{i=1}^{\infty}a_ie_i\right\|
\end{equation}
holds.
\end{definition}

\begin{theorem}[\cite{BS75}] For each nonreflexive space $X$ there is
a Banach space $E$ with an \ESA\ basis which is finitely
representable in $X$.
\end{theorem}

Since this theorem is not explicitly stated in \cite{BS75}, we
describe how to get it from the argument presented there. By
\cite{Pta59}, there is a sequence $\{x_i\}_{i=1}^\infty$ in $B_X$
(the unit ball of $X$) satisfying, for some $0<\theta<1$ and some
$\{f_i\}_{i=1}^\infty\subset B_{X^*}$ the condition
\[f_n(x_k)=\begin{cases} \theta &\hbox{ if }n\le k\\
0 &\hbox{ if }n>k.
\end{cases}\]
Following \cite[Proposition 1]{BS74} we build on the sequence
$\{x_i\}$ the spreading model $\widetilde X$ (the term {\it
spreading model} was not used in \cite{BS74}, it was introduced
later, see \cite[p.~359]{Bea79}). The natural basis
$\{e_i\}_{i=1}^\infty$ in $\widetilde X$ is {\it invariant under
spreading} (IS) in the sense that
\[\left\|\sum_i \alpha_ie_{k_i}\right\|=\left\|\sum_i
\alpha_ie_{i}\right\|\] for each strictly increasing sequence
$\{k_i\}$ of positive integers. The space $\widetilde X$ is
finitely representable in $X$, see \cite[p.~83]{BS75}. Now we use
the procedure described in \cite[p.~84]{BS75}, and get a Banach
space $E$ which is finitely representable in $\widetilde X$ and
has an \ESA\ basis. (Actually, the fact that we get a basis was
not verified in \cite{BS75}, this was done in \cite[Proposition
1]{BS76}).
\medskip

Since the space $E$ has nontrivial type, it follows from results
of \cite[Lemma 3, p.~290]{BS76} that this basis is boundedly
complete and hence $E$ is isomorphic to a dual space (see
\cite[Proposition 1.b.4]{LT77}).

\begin{remark} It would be interesting to show that $E$ is
isometric to a dual space. Then we would be able to omit the word
`crudely' from the statement of Theorem \ref{T:RNPandSuperRefl}.
\end{remark}

Let $R$ be a Banach space such that $R^*$ is isomorphic to $E$. We
construct the desired space $Z$ as a transfinite dual of $R$.
Transfinite duals were introduced in \cite{DJL76}. Let us recall
the definition. We denote the $n$th dual $(n\in \mathbb{N})$ of a
Banach $R$ by $R^{(n)}$. We say that an ordinal $\alpha$ is {\it
even} if it is either a limit ordinal or an ordinal of the form
$\beta+2n$ where $\beta$ is a limit ordinal and $n\in \mathbb{N}$.
We denote $R^{(\alpha)}$ by transfinite induction:

\begin{itemize}

\item $R^{(\alpha+1)}=(R^{(\alpha)})^*$.

\item If $\alpha$ is a limit ordinal, we let $R^{(\alpha)}$ to be
the completion of the union
\[\bigcup_{\stackrel{\beta<\alpha}{\beta{\rm~ is~ even}}}R^{(\beta)}.\]
(Observe that the union is well defined as a normed linear space
since $R^{(\beta)}$ admits a canonical isometric embedding into
$R^{(\gamma)}$ if $\beta<\gamma$ and both $\beta$ and $\gamma$ are
even.)
\end{itemize}
To complete the proof of Theorem \ref{T:RNPandSuperRefl} we prove
the following two statements:

\begin{enumerate}

\item\label{I:CFR} The space $R^{(\omega^2+1)}$ is crudely
finitely representable in $E$ (and thus in $X$).

\item The space $R^{(\omega^2)}$ is separable and the space
$R^{(\omega^2+1)}$ is nonseparable.

\end{enumerate}

Statement \ref{I:CFR} is an immediate consequence of the following
lemma:

\begin{lemma} Let $R$ be a Banach space and $R^*(=R^{(1)})$ be its
dual. Then $R^{(\gamma)}$ is finitely representable in $R^*$ for
every odd ordinal $\gamma$.
\end{lemma}

\begin{proof} For finite ordinals this result is an immediate consequence
of the local reflexivity principle \cite{LR69}. The same principle
implies that if the statement is true for an infinite odd ordinal
$\gamma$, then it is true for all ordinals of the form
$\gamma+2n$. So if we use the transfinite induction, it remains to
show that the statement holds for ordinals of the form
$\gamma=\alpha+1$, where $\alpha$ is a limit ordinal, provided it
holds for all smaller odd ordinals.
\medskip

We have

\begin{equation}\label{E:Zbeta}
R^{(\alpha)}=\cl\left(\bigcup_{\stackrel{\beta<\alpha}{\beta{\rm~
is~ even}}} R^{(\beta)}\right).\end{equation}

Let $F$ be a finite dimensional subspace of $R^{(\alpha+1)}$,
$\ep>0$. Let $\{f_i\}_{i=1}^k$ be a finite $\frac{\ep}2$-net in
$S_F$ (the unit sphere of $F$). For each $f_i$ we can find an even
ordinal $\beta_i<\alpha$ and a vector $x_i\in Z^{(\beta_i)}$ such
that $||x_i||=1$ and $f_i(x_i)\ge 1-\frac{\ep}2$. Let
$\tau=\max_{1\le i\le k}\beta_i$. Then the natural restriction of
$F$ to the space $R^{(\tau)}$ is an $\ep$-isometry, hence $F$ is
$\ep$-isometric to a subspace in $R^{(\tau+1)}$, and the induction
hypothesis implies that $R^{(\alpha+1)}$ is finitely representable
in $R^*$.
\end{proof}

To show that $R^{(\omega^2)}$ is separable it suffices to show
that $R^{(n\omega)}$ is separable for each $n$. This can be shown
by a straightforward induction based on the following results:

\begin{theorem}[{\cite[Theorem 16]{Per79}}] If $X$ is quasireflexive, then
$X^{(\omega)}=X\oplus[x_i]$, where $\{x_i\}$ is an ESA basis.
\end{theorem}

\begin{theorem}[{\cite[Theorem 3]{BS76}}] If a Banach space with an ESA basis has
nontrivial type, then it is quasireflexive. \end{theorem}

The fact that $R^{(\omega^2+1)}$ is nonseparable was proved by
Bellenot \cite{Bel82}. Since the details of the argument of
Bellenot are difficult to follow, we would like to mention that
this result can be derived using the argument of Davis and
Lindenstrauss \cite[Theorem 4]{DL76}. Let us mention the
modification of the argument of \cite{DL76} needed to achieve this
goal. To understand the discussion below the reader is expected to
read the very elegant proof in \cite[pp.~194--196]{DL76} (we would
like to mention that there are two misprints on page 195, line 15
from above: $f_{(\sigma,n)}$ should be $f_{(0,n)}$ and $f_{(0,n)}$
should be $f_{(1,n)}$).

We build the collections $x_{(\sigma,n)}$ and $f_{(\sigma,n)}$ in
the way described in \cite[pp.~194--195]{DL76}. Then, for each
$\sigma$ in the Cantor set $\Delta$ we pick a sequence
$\{\sigma_j\}_{j=1}^\infty$ of end points in $\Delta$ so that
$\sigma_j\to\sigma$ and let $F_\sigma\in R^{(\omega^2+1)}$ be any
weak$^*$ limit point of the sequence
$\{f_{(\sigma_j,1)}\}_{j=1}^\infty$ in $R^{(\omega^2+1)}$. We
claim that $||F_\sigma-F_\tau||\ge\frac12$ for each
$\sigma,\tau\in\Delta$, $\sigma<\tau$. The reason is: we can find
a $\lambda$ which is an end point of $\Delta$ and satisfies
$\sigma<\lambda<\tau$. But then, as is easy to check, for any
$n\in\mathbb{N}$ we have $F_\sigma (x_{(\lambda,n)})=1$ and
$F_\tau (x_{(\lambda,n)})=0$. Since $||x_{(\lambda,n)}||\le 2$,
the conclusion follows. \hfill $\Box$

\section{Proof of Proposition \ref{P:FinToInf}}\label{S:HilbEx}

The fact that finite subsets of $\ell_2$ admit embeddings into an
arbitrary non-superrefle\-xi\-ve Banach space with uniformly
bounded distortions is an immediate consequence of the Dvoretzky
theorem \cite{Dvo61}.

As an example of a suitable space $X$ we use the James tree space
(see \cite{Jam74,LS75}), but build on $\ell_p$ with
$p\in(2,\infty)$. More precisely we follow the construction of
\cite[Section 2]{LS75}. So we consider an infinite binary tree
$T_\infty$ whose vertices can be labelled with finite sequences of
$0$s and $1$s (including the empty sequence) with the norm
\[||x||=\sup\left(\sum_{j=1}^k\left(\sum_{v\in
\mathcal{J}_j}x(v)\right)^p\right)^{\frac1p}<\infty,\] where the
supremum is taken over all choices of $k$ and of pairwise disjoint
finite descending paths $\mathcal{J}_1,\dots,\mathcal{J}_k$ in the
tree $T_\infty$. Denote by $B$ the closed linear span in
$(JT_p)^*$ of the biorthogonal functionals $\{e_v^*\}$ of the unit
vector basis $\{e_v\}$ of $JT_p$.
\medskip

In the same way as in \cite[Section 2]{LS75} one can establish the
following results:

\begin{enumerate}

\item $JT_p$ is naturally isomorphic to $B^*$.

\item The quotient of $(JT_p)^*$ with the kernel $B$ is isometric
to $\ell_q(\Gamma)$ where $\Gamma$ is a set of cardinality
continuum and $\frac1q+\frac1p=1$.

\item The space $(JT_p)^*$ is a nonseparable dual of a separable
Banach space.

\end{enumerate}

It remains to prove that $(JT_p)^*$ does not admit a bilipschitz
embedding of $\ell_2$. In fact, otherwise, by \cite[Corollary
7.10]{BL00}, it would contain a linear isomorphic image of
$\ell_2$. Since $\ell_q(\Gamma)$ with $q\in(1,2)$ is totally
incomparable with $\ell_2$, this would imply that $B$ contains a
subspace isomorphic to $\ell_2$. This can be shown to be false in
the following way.
\medskip

Assume that $B$ contains a sequence $\{b_i\}_{i=1}^\infty$
equivalent to the unit vector basis of $\ell_2$. Clearly we may
assume that $\{b_i\}_{i=1}^\infty$ is disjointly supported with
respect to the basis $\{e_v^*\}$. Let
$\{b_i^*\}_{i=1}^\infty\subset JT_p$ be a bounded sequence
satisfying $b_i^*(b_i)=1$. The sequence $\{b_i^*\}$ also can be
assumed to be disjointly supported. By \cite[Corollary 3]{LS75}
(see also \cite[Proposition on p.~91]{LS75}), we may assume that
$\{b_i^*\}$ is weakly Cauchy. Then the sequence
$\{b^*_{2k}-b^*_{2k-1}\}_{k=1}^\infty$ is weakly null. Using a
straightforward generalization of \cite[Theorem, p.~420]{AI81} we
get that $\{b^*_{2k}-b^*_{2k-1}\}_{k=1}^\infty$ contains a
subsequence equivalent to the unit vector basis of $\ell_p$. We
assume that $\{b^*_{2k}-b^*_{2k-1}\}_{k=1}^\infty$ is equivalent
to the unit vector basis of $\ell_p$. Then, as is easy to see, we
get that for some constant $c>0$ and any finitely non-zero
sequence $\{\alpha_k\}$ we have
\[\left\|\sum_k\alpha_kb_{2k}\right\|\ge
c\left(\sum_k\alpha^q_{k}\right)^{\frac1q}.\] Since $q\in(1,2)$,
this contradicts to the assumption that $\{b_n\}$ is equivalent to
the unit vector basis of $\ell_2$. \hfill $\Box$

\end{large}

\begin{small}

\renewcommand{\refname}{
\end{small}

\end{document}